\documentclass[11pt]{amsart}

\usepackage{amsmath}
\usepackage{amsthm}
\usepackage{amssymb}
\usepackage{amsrefs}
\usepackage{todonotes}
\usepackage{enumitem}

\usepackage[colorlinks=true]{hyperref}

\usepackage[mathlines,modulo]{lineno}

\theoremstyle{plain}
\newtheorem{theorem}{Theorem}[section]
\newtheorem{proposition}[theorem]{Proposition}
\newtheorem{lemma}[theorem]{Lemma}
\newtheorem{corollary}[theorem]{Corollary}

\theoremstyle{definition}
\newtheorem{definition}[theorem]{Definition}
\newtheorem{example}[theorem]{Example}

\newcommand{\formula}[1]{\ensuremath{\text{\fontfamily{cmss}\selectfont \textup{#1}}}}
\newcommand{\indep}[1]{\ensuremath{\formula{Indep}[#1]}}
\newcommand{\sing}[1]{\ensuremath{\formula{Sing}[#1]}}
\newcommand{\sfmax}{\ensuremath{\formula{max}}}
\newcommand{\sfmin}{\ensuremath{\formula{min}}}

\newcommand{\dash}{\nobreakdash-\hspace{0mm}}
\newcommand{\cmso}{\ensuremath{\mathit{CMS}_{0}}}
\newcommand{\mso}{\ensuremath{\mathit{MS}_{0}}}
\newcommand{\mstwo}{\ensuremath{\mathit{MS}_{2}}}
\newcommand{\ds}{\displaystyle}

\title[Matroid transductions]{Monadic transductions and definable classes of matroids}

\author[Jowett]{Susan Jowett}
\author[Mayhew]{Dillon Mayhew}
\author[Mo]{Songbao Mo}
\author[Tuffley]{Christopher Tuffley}

\begin{document}

\begin{abstract}
A transduction provides us with a way of using the monadic second-order language of a structure to make statements about a derived structure.
Any transduction induces a relation on the set of these structures.
This article presents a self-contained presentation of the theory of transductions for the monadic second-order language of matroids.
This includes a proof of the matroid version of the Backwards Translation Theorem, which lifts any formula applied to the images of the transduction into a formula which we can apply to the pre-images.
Applications include proofs that the class of lattice-path matroids and the class of spike-minors can be defined by sentences in monadic second-order logic.
\end{abstract}

\maketitle

\section{Introduction}  

The monadic second-order logic of discrete objects derives its importance from connections with the theory of algorithms.
This is exemplified by Courcelle's foundational theorem~\cite{MR1042649} and its descendants \cites{MR2066603, MR4444152}.
Transductions play an essential role in the study of discrete structures and their monadic second-order logics.
They are central to the work of Boja\'{n}czyk and Pilipczuk~\cite{MR3776760} that resolved a long-standing conjecture of Courcelle's~\cite{MR1042649}.
In the context of the monadic second-order logic of matroids, transductions have been implicitly used by Funk, Mayhew, and Newman~\cites{MR4444152,MR4395073}.

Roughly speaking, a transduction is a way of talking about a derived structure by using a logical language that applies to the original structure.
To illustrate the idea of a transduction, we consider an example.
Suppose that we want to speak about the minors of the graph $G$ by using the monadic second-order logic \mstwo\ for graphs, applied to $G$.
Any minor of $G$ can be produced by contracting a forest $F$ of $G$ and then deleting a set $D$ of edges.
Let $H$ be the minor produced in this way.
Now the vertices of $H$ correspond to equivalence classes of a relation on the vertices of $G$.
Two vertices are related if and only if they are joined by a path of $G[F]$.
There is an \mstwo\ formula that is true for exactly these equivalence classes, so the sets satisfying this formula become the vertices of $H$.
A set of vertices of $H$ is a union of these equivalence classes.
An edge of $H$ is an edge of $G$ that is not in $F\cup D$.
An edge is incident with a vertex of $H$ if and only if that edge of $G$ is incident with a vertex in the equivalence class.
In this way we see that by quantifying over all appropriate choices of $F$ and $D$, we are able to use the \mstwo\ language of $G$ to make statements about all the minors of $G$.
We view this transduction as a relation taking each graph to its minors.

Transductions are important because any sentence that applies to the images of the transduction can be lifted into a formula that applies to the pre-images.
We illustrate this with another example: the transduction that takes any graph to its connected components.
The vertex set of a connected component is an equivalence class under a relation on the vertices of $G$.
(In this case, two vertices are related if there is a walk of $G$ between them.)
Once again, there is a formula in \mstwo\ that is satisfied by the equivalence classes of this relation.
If $H$ is such an equivalence class, we can make statements about $H$ by letting the vertices be the members of $H$ and letting the edges be the edges of $G$ that join two vertices in $H$.
By quantifying over all equivalence classes, we are able to make statements about all the connected components of $G$.
To see why this is useful, imagine that some property of graphs is closed under disjoint unions, and there is an \mstwo\ sentence that characterises the property for connected graphs.
We apply the aforementioned lifting operation to this sentence.
By appropriately quantifying over the formula that we obtain, we derive a sentence that is true for $G$ if and only if each connected component of $G$ has the property in question.
This application of transductions shows us that if we want to construct a sentence characterising the graph property, it is sufficient to find a sentence that does so for connected graphs.

The idea we have just described is the content of the \emph{Backwards Translation Theorem}, which is the key tool used in~\cite{MR3776760}.
In Theorem~\ref{BTT} we prove a version of the Backwards Translation Theorem that is specialised to matroids.
This theorem formalises the process of lifting a formula into a new formula that applies to pre-images.

Transductions and the Backwards Translation Theorem are covered extensively in the colossal work of Courcelle and Engelfreit~\cite{MR2962260}.
Because their work is developed at a high level of abstraction it is often useful to consider versions of transductions that apply to more specific contexts.
Thus, a principal aim of the present work is to set out the theory of transductions for the (counting) monadic second-order logic of matroids in a self-contained exposition.

The paper is organised as follows.
In Section~\ref{sect:prelims} we gather together the fundamental definitions in the monadic theory of matroids and collate useful formulas.
In Section~\ref{sect:transductions} we motivate and define monadic transductions for matroids.
We also list several natural and useful matroid transductions and prove the Backwards Translation Theorem.
In Section~\ref{sect:operations} we collect tools that we can use to prove definability results for classes of matroids.
For example, we prove that a minor-closed class of matroids is monadically definable if and only if the class of excluded minors can be monadically defined (Proposition~\ref{definable-excluded}).
Finally, in  Section~\ref{sect:classes}, we apply ideas from the rest of the paper to prove that some natural matroid classes are definable.
In particular we show that there are monadic characterisations of the class of lattice-path matroids~\cite{MR2018421} and the class of spikes (introduced in~\cite{MR1399683}).
The former class is of interest in this context because it is monadically definable despite the fact that it has infinitely many excluded minors~\cite{MR2718679}.
(Any minor-closed class with finitely many excluded minors is monadically definable by an observation originally due to Hlin\v{e}n\'{y}~\cite{MR2081597}.)

We remark here that our definition of transduction does not use the copying and colouring operations used in~\cite{MR3776760}, so perhaps we should more properly use the term \emph{interpretation}.
We think our use of ``transduction" will not cause any confusion.

\section{Preliminaries}
\label{sect:prelims}

A \emph{set-system} is a pair $(E,\mathcal{I})$, where $E$ is a finite set and $\mathcal{I}$ is a collection of subsets of $E$.
We refer to $E$ as the \emph{ground set} and the members of $\mathcal{I}$ as \emph{independent sets}.

We introduce the logical language \cmso.
The variables of the language are $X_{1},X_{2},X_{3},\ldots$, where each variable is going to be interpreted as a subset of a ground set.
We freely use other symbols such as $T$, $U$, $V$, $W$, $X$, $Y$, and $Z$  to represent members of this family of variables.
We have the predicate symbols $\indep{\cdot}$, $\subseteq$, and $|\cdot|_{p,q}$, where $p$ and $q$ are any non-negative integers such that $p < q$.
The \emph{formulas} of \cmso\ are defined recursively.
Each formula contains a number of variables, and each such variable is either \emph{free} or \emph{bound} in the formula.
The following are \emph{atomic formulas}:
\begin{enumerate}[label = \textup{(\roman*)}]
\item the unary formula $\indep{X}$, which has the single (free) variable $X$,
\item the binary formula $X\subseteq Y$, which has $X$ and $Y$ as (free) variables, and
\item the unary formula $|X|_{p,q}$, where $p$ and $q$ are non-negative integers such that $p<q$.
This formula has $X$ as its single (free) variable.
\end{enumerate}
Now we define non-atomic formulas.
If $\varphi$ is a formula, then so is $\neg \varphi$, and these two formulas have the same variables.
A variable is free in $\neg\varphi$ if and only if it is free in $\varphi$.
If $X$ is a free variable in the formula $\varphi$, then $\exists X\varphi$ is a formula with the same variables as $\varphi$.
We say that $\exists$ is the \emph{existential quantifier}.
A variable is free in $\exists X\varphi$ if and only if it is free in $\varphi$, with the exception of $X$ which is a bound variable of $\exists X\varphi$.
Finally, if $\varphi_{1}$ and $\varphi_{2}$ are formulas, and no variable is free in one of $\varphi_{1}$ and $\varphi_{2}$ while being bound in the other, then $\varphi_{1}\land \varphi_{2}$ is a formula.
Since we can always rename bound variables, the constraint on variables imposes no difficulty.
The variables of $\varphi_{1}\land \varphi_{2}$ are those in either $\varphi_{1}$ or $\varphi_{2}$, and the free variables of $\varphi_{1}\land \varphi_{2}$ are exactly those variables that are free in either $\varphi_{1}$ or $\varphi_{2}$.

We will use parentheses freely to clarify the construction of formulas.
A formula with no free variables is a \emph{sentence}.

\begin{definition}
The collection of formulas constructed in this way is \emph{counting monadic second-order logic} (\cmso).
The collection of formulas that do not use any atomic formula of the form $|\cdot|_{p,q}$ is \emph{monadic second-order logic} (\mso).
\end{definition}

\subsection{Shorthands}

We use various pieces of shorthand notation.
The formula $X=Y$ means $(X\subseteq Y)\land (Y\subseteq X)$, and $X\ne Y$ means $\neg(X=Y)$.
We write $X\nsubseteq Y$ for $\neg(X\subseteq Y)$.
If $\varphi_{1}$ and $\varphi_{2}$ are formulas with no variable bound in one and free in the other, then $\varphi_{1}\lor\varphi_{2}$ means $\neg(\neg \varphi_{1} \land \neg \varphi_{2})$, and $\varphi_{1}\to \varphi_{2}$ means $(\neg \varphi_{1})\lor \varphi_{2}$, while $\varphi_{1}\leftrightarrow\varphi_{2}$ means $(\varphi_{1}\to\varphi_{2})\land (\varphi_{2}\to\varphi_{1})$.
If $X$ is a free variable in the formula $\varphi$, then $\forall X\varphi$ is shorthand for $\neg(\exists X (\neg\varphi))$.
We call $\forall$ the \emph{universal quantifier}.

We write $\varphi[X_{i_{1}},\ldots, X_{i_{s}}]$ to indicate that $\varphi$ is a \cmso\dash formula and that its free variables are $X_{i_{1}},\ldots, X_{i_{s}}$.
In this case,
\[\varphi[Z_{i_{1}},\ldots, Z_{i_{s}}]\]
denotes the formula that we obtain from $\varphi$ by replacing every occurrence of the variable $X_{i_{j}}$ with the variable $Z_{i_{j}}$, for each $j$.
If $\mathcal{X}$ is the tuple $(X_{i_{1}},\ldots, X_{i_{s}})$ we may write $\varphi[\mathcal{X}]$ instead of $\varphi[X_{i_{1}},\ldots, X_{i_{s}}]$.
If $\varphi$ has $X$ as a free variable, then we may write $\varphi[X\rightharpoonup Z]$ to denote the formula obtained from $\varphi$ by replacing each occurrence of $X$ with $Z$.

\subsection{Interpretations}

Let $\varphi$ be an \cmso\dash formula and let $M=(E,\mathcal{I})$ be a set-system.
An \emph{interpretation} of $\varphi$ is a function, $\theta$, from the set of free variables of $\varphi$ to subsets of $E$.
We treat every function as a set of ordered pairs.
We will blur the distinction between a variable and its image under an interpretation when it is convenient to do so.

We recursively define what it means for $(M,\theta)$ to \emph{satisfy} $\varphi$.
If $\varphi$ is the atomic formula $\indep{X}$, then $(M,\theta)$ satisfies $\varphi$ if and only if $\theta(X)\in\mathcal{I}$.
If $\varphi$ is the atomic formula $X\subseteq Y$, then $(M,\theta)$ satisfies $\varphi$ if and only if $\theta(X)\subseteq \theta(Y)$.
Next, if $\varphi$ is $|X|_{p,q}$, then $(M,\theta)$ satisfies $\varphi$ if and only if $|\theta(X)|$ is congruent to $p$ modulo $q$.

Now we move to formulas that are not atomic.
Assume that $\varphi=\neg\psi$, for some formula $\psi$.
Then $(M,\theta)$ satisfies $\varphi$ if and only if $(M,\theta)$ does not satisfy $\psi$.
Next let $\varphi$ be $\varphi_{1}\land \varphi_{2}$.
For $i=1,2$, let $\theta_{i}$ be the restriction of $\theta$ to the free variables of $\varphi_{i}$.
Then $(M,\theta)$ satisfies $\varphi$ if and only if $(M,\theta_{1})$ satisfies $\varphi_{1}$ and $(M,\theta_{2})$ satisfies $\varphi_{2}$.
Finally, assume that $\varphi=\exists X\psi$.
Then $(M,\theta)$ satisfies $\varphi$ if and only if there is a subset, $F\subseteq E$, such that the interpretation $(M,\theta\cup\{(X,F)\})$ satisfies $\psi$.

Instead of saying that $(M,\theta)$ satisfies $\varphi$, we may say that $\varphi$ is satisfied by $M$ under the interpretation $\theta$.
If $\varphi$ is a sentence, meaning it has no free variables, then we will say that $M$ satisfies $\varphi$.
If $\varphi$ has free variables $X_{i_{1}},\ldots, X_{i_{s}}$, and $\theta$ is a satisfying assignment taking these variables to subsets $Y_{i_{1}},\ldots, Y_{i_{s}}$, then we may say that $Y_{i_{1}},\ldots, Y_{i_{s}}$ satisfy $\varphi$.

\begin{definition}
Let $\mathcal{M}$ be a class of set-systems.
If there is an \mso/\cmso\dash sentence $\varphi$ such that $M=(E,\mathcal{I})$ satisfies $\varphi$ if and only $M$ is in $\mathcal{M}$, then $\mathcal{M}$ is \emph{(\mso/\cmso)-definable}.
\end{definition}

\subsection{Matroid formulas.}

The next result is proved in~\cite{MR3803151}.
    
\begin{proposition}
There exists an \mso\dash sentence \formula{Matroid} that is satisfied by a set-system $M=(E,\mathcal{I})$ if and only if $\mathcal{I}$ is the collection of independent sets of a matroid on the ground set $E$.
\end{proposition}

Let $\formula{Empty}[X]$ be the formula
\[
\forall Z\ (Z\subseteq X\to Z=X).
\]
Note that $X$ is the only free variable of $\formula{Empty}[X]$, and that a satisfying interpretation must take $X$ to a subset that is equal to its only subset: that is, the empty set.
Let $\formula{Sing}[X]$ be the formula
\[
\neg \formula{Empty}[X]\land \forall Z\ (Z\subseteq X\to (\formula{Empty}[Z]\lor Z=X)).
\]
Then $\formula{Sing}[X]$ is satisfied by precisely the interpretations that take $X$ to a set of cardinality one.

Let $\formula{Union}[X,Y,Z]$ be the formula
\[
X\subseteq Z\land Y\subseteq Z\land (\forall W\ (X\subseteq W\land Y\subseteq W\to Z\subseteq W)).
\]
Note that $\formula{Union}[X,Y,Z]$ is satisfied if and only if $Z$ is interpreted as the union of the interpretations of $X$ and $Y$.
Similarly, let $\formula{Intersection}[X,Y,Z]$ be the formula
\[
Z\subseteq X\land Z\subseteq Y\land (\forall W\ ((W\subseteq X\land W\subseteq Y)\to W\subseteq Z).
\]
Then $\formula{Intersection}[X,Y,Z]$ is satisfied if and only if $Z$ is interpreted as the intersection of the interpretations of $X$ and $Y$.
Let $\formula{Disjoint}[X,Y]$ be the formula
\[
\forall Z\ ((Z\subseteq X \land Z\subseteq Y)\to \formula{Empty}[Z]).
\]
Then $\formula{Disjoint}[X,Y]$ is satisfied if and only if $X$ and $Y$ are interpreted as disjoint sets.
Let $\formula{Bipartition}[X,Y]$ be
\[\formula{Disjoint}[X,Y]\land \forall W\ (\formula{Sing}[W]\to W\subseteq X\lor W\subseteq Y).\]

Let $\varphi[X]$ be a formula with a single free variable, $X$.
Let $M=(E,\mathcal{I})$ be a set-system and let $\mathcal{D}$ be the collection of subsets of $E$ such that $D\subseteq E$ belongs to $\mathcal{D}$ if and only if $\varphi[X]$ is satisfied by the interpretation that takes $X$ to $D$.
Now let $\sfmax_{\varphi}[X]$ be the formula
\[
\varphi[X] \land \forall Y\ ((\varphi[Y] \land X\subseteq Y)\to Y=X).
\]
Then $\sfmax_{\varphi}[X]$ is satisfied by the interpretation taking $X$ to $D$ if and only if $D\in \mathcal{D}$ and $D$ is not properly contained in any member of $\mathcal{D}$; that is, if and only if $D$ is a maximal member of $\mathcal{D}$.
In the same way,
\[
\sfmin_{\varphi}[X]=\varphi[X] \land \forall Y\ ((\varphi[Y] \land Y\subseteq X)\to Y=X)
\]
is satisfied by the minimal members of $\mathcal{D}$.

We define $\formula{Basis}[X]$ to be the formula $\sfmax_{\formula{Indep}}[X]$.
Let $\formula{Dep}[X]$ be $\neg\formula{Indep}[X]$, and let \formula{Circuit} be $\sfmin_{\formula{Dep}}[X]$.
If $M=(E,\mathcal{I})$ is a matroid then $\formula{Basis}[X]$ will be satisfied by exactly those interpretations that take $X$ to a basis of $M$, and $\formula{Circuit}[X]$ will be satisfied by exactly those interpretations that take $X$ to a circuit.

We let $\formula{Coindep}[X]$ be the formula $\exists B\ (\formula{Basis}[B]\land \formula{Disjoint}[B,X])$.
Assuming that $M$ is a matroid, then $X$ satisfies $\formula{Coindep}[X]$ if and only if the complement of $X$ contains a basis: that is, if and only if $X$ is coindependent in $M$.
We set $\formula{Codep}[X]$ to be $\neg\formula{Coindep}[X]$.
We also let $\formula{Cocircuit}[X]$ be $\formula{min}_{\formula{Codep}}[X]$, so that $\formula{Cocircuit}[X]$ is satisfied by the cocircuits of $M$.

Next we design the formula $\formula{Flat}[X]$ so that when $M$ is a matroid, the formula will be satisfied when $X$ is a flat.
We can achieve this by making $\formula{Flat}[X]$ be
\begin{multline*}
\forall C\ (\formula{Circuit}[C]\land \exists Y\ (\formula{Sing}[Y]\land Y\subseteq C\land Y\nsubseteq X))\to\\
\exists Z\ (\formula{Sing}[Z]\land Z\subseteq C\land Z\nsubseteq X\land Z\ne Y)).
\end{multline*}
This says that there is no circuit that contains exactly one element not in $X$, which is equivalent to $X$ being a flat.
By replacing the $\formula{Circuit}$ formula in $\formula{Flat}$ with $\formula{Cocircuit}$, we obtain the formula $\formula{Coflat}[X]$, which will be satisfied when $X$ is a flat in the dual matroid.
The formula $\formula{Spanning}[X]$ is $\exists B\ (\formula{Basis}[B]\land B\subseteq X)$ and $\formula{Hyperplane}[X]$ is $\formula{max}_{\neg\formula{Spanning}}[X]$.

If $X$ and $Y$ are disjoint sets in a matroid $M$, we say that $X$ and $Y$ are \emph{skew} if there is no circuit of $M$ contained in $X\cup Y$ that contains elements from both $X$ and $Y$.
There is a corresponding formula $\formula{Skew}[X,Y]$ which is equal to
\begin{multline*}
\formula{Disjoint}[X,Y]\land \neg\exists C\ (\formula{Circuit}[C]\land \neg\formula{Disjoint}[C,X]\land \\
\neg\formula{Disjoint}[C,Y]\land \forall Z\ (\formula{Union}[X,Y,Z]\to C\subseteq Z)).
\end{multline*}
By replacing $\formula{Circuit}$ with $\formula{Cocircuit}$ we obtain the formula $\formula{Coskew}[X,Y]$.

We define the formula $\formula{Separator}[X]$ to be
\[\neg\formula{Empty}[X]\land\forall C\ (\formula{Circuit}[C]\to (C\subseteq X\lor \formula{Disjoint}[C,X])).
\]
When $M$ is a matroid, this formula is satisfied if and only if there is no circuit that contains elements of both $X$ and the complement of $X$.
This is exactly what it means for $X$ to be a separator of the matroid.
A minimal separator is a connected component, so we set the formula $\formula{Component}[X]$ to be $\formula{min}_{\formula{Separator}}[X]$.
A separator is a $1$\dash separating set.
More generally, the proof of Proposition~3.2 in~\cite{MR4395073} shows that for each positive integer $k$ there is a formula $k$\dash$\formula{Separating}[X]$ such that when $M$ is a matroid, the formula will be satisfied exactly when $X$ is $k$\dash separating, which is to say, when $r(X)+r(E(M)-X)-r(M)<k$.

\section{Transductions}
\label{sect:transductions}

Informally, a transduction is a way of expressing statements about a derived structure by using the logical language associated to the original structure.
For example, we may wish to make statements about a dual matroid $M^{*}$ by using the logical language of the original matroid $M$.
We can do this by replacing every occurrence of ``$X$ is independent" with ``$X$ is disjoint from some basis".
This has the effect of replacing every statement that $X$ is independent with a statement saying that $X$ is coindependent.
If the sentence $\varphi$ defines a class of matroids $\mathcal{M}$, then performing these replacements produces a sentence that defines the dual class $\{M^{*}\colon M\in \mathcal{M}\}$.
In other words, if a class is definable, then so is its dual class.
This is a consequence of the \emph{Backwards Translation Theorem}, which we now start to develop.

\subsection{The definition}

First we formally define a transduction.

\begin{definition}
An \emph{\mso\dash transduction} is a triple of formulas:
\[
\Lambda= (\formula{Domain}, \formula{NewElement}, \formula{NewIndep}),
\]
where the set of free variables of \formula{Domain} is $\mathcal{Z}$ and both \formula{NewElement} and \formula{NewIndep} contain the free variable $T$, and every other free variable of these formulas is in $\mathcal{Z}$.
(We insist that $T$ is not in $\mathcal{Z}$.)
Note that there may be variables in $\mathcal{Z}$ that do not appear in \formula{NewElement} or \formula{NewIndep}.

From $\Lambda$ we derive a relation taking each set-system $M=(E,\mathcal{I})$ to other set-systems.
We first describe this relation informally.
Each interpretation of $\mathcal{Z}$ that satisfies \formula{Domain} will define one output of the relation.
Let $\rho\colon \mathcal{Z}\to 2^{E}$ be such a satisfying interpretation.
Then the output will be $M_{\rho}=(E_{\rho},\mathcal{I}_{\rho})$, where $E_{\rho}$ is the collection of subsets of $E$ that satisfy \formula{NewElement}.
If a union of these subsets satisfies \formula{NewIndep}, then that collection of subsets will be in $\mathcal{I}_{\rho}$.

Let us now make these ideas more formal.
We again assume that $\rho$ is a function from $\mathcal{Z}$ to the power set of $E$ such that $(M,\rho)$ satisfies \formula{Domain}.
Let $\mathcal{X}$ be the set of variables in $\mathcal{Z}$ that appear as free variables in \formula{NewElement}.
We define $E_{\rho}$ to be the set of subsets $F\subseteq E$ such that the function
\[\rho|_{\mathcal{X}}\cup\{(T,F)\}\]
is a satisfying interpretation of \formula{NewElement}.
Note that $E_{\rho}$ is a collection of subsets of $E$.
Let $I$ be any subset of $E_{\rho}$, so that $I$ is also a collection of subsets of $E$.
We let $\Lambda^{-1}_{\rho}(I)$ stand for the union $\cup_{Y\in I} Y$.
Thus $\Lambda^{-1}_{\rho}(I)$ is a subset of $E$.

Now let $\mathcal{Y}$ be the set of variables in $\mathcal{Z}$ that appear as free variables in \formula{NewIndep}.
Let $\mathcal{I}_{\rho}$ be the collection of subsets of $E_{\rho}$ such that $I$ belongs to $\mathcal{I}_{\rho}$ if and only if the interpretation $\rho|_{\mathcal{Y}}\cup\{(T,\Lambda^{-1}_{\rho}(I))\}$ satisfies \formula{NewIndep}.
We define $M_{\rho}$ to be the set-system $(E_{\rho},\mathcal{I}_{\rho})$.

We define the relation $\mathcal{T}_{\Lambda}$ to include all pairs $(M,M_{\rho})$ where $M$ ranges over all set-systems $(E,\mathcal{I})$, and $\rho$ ranges over all satisfying interpretations of \formula{Domain}.
We use $\mathcal{T}_{\Lambda}(M)$ to stand for the collection $\{M_{\rho}\colon (M,M_{\rho})\in \mathcal{T}_{\Lambda}\}$.

These conditions are sufficient to prove that the Backwards Translations Theorem holds for \mso\dash formulas.
But if we want it to apply to formulas that use modular counting predicates, we must impose an additional condition.
If every member of $E_{\rho}$ is a singleton set, for every $M=(E,\mathcal{I})$ and every interpretation $\rho\colon \mathcal{Z}\to 2^{E}$ that satisfies \formula{Domain}, then we say that $\Lambda$ is a \emph{\cmso\dash transduction}.
\end{definition}

Let us consider a \cmso\dash transduction $\Lambda$.
Assume that $M=(E,\mathcal{I})$ is a set-system and $(E_{\rho},\mathcal{I}_{\rho})$ is a set-system in $\mathcal{T}_{\Lambda}(M)$.
Now, as every member of $E_{\rho}$ is a singleton subset of $E$, we can identify each member of $E_{\rho}$ with an element of $E$. Thus we can think of $E_{\rho}$ as being a subset of $E$, and we can think of $\mathcal{I}_{\rho}$ as a collection of subsets of $E$.
We freely make this identification whenever it improves readability.

\subsection{Examples}

\begin{example}
\label{exampledual}
We will construct a \cmso\dash transduction taking each matroid to its dual.
Let \formula{Domain} be the sentence \formula{Matroid}, so $M=(E,\mathcal{I})$ has an output only if it is a matroid.
Note that in this case \formula{Domain} has no free variables, so $\mathcal{Z}$ is empty.
Let $\formula{NewElement}[T]$ be \sing{T}, which means that the ground set of the output will be $\{\{e\}\colon e\in E\}$ when the input is $M=(E,\mathcal{I})$.
Thus we can canonically associate the elements in the ground set of the output with the elements in $E$.
As discussed above, we blur this distinction, and think of the ground set of the output as being exactly the same as the ground set of the input to the transduction.
(This is helpful, because a matroid and its dual have the same ground set.)

Next we define $\formula{NewIndep}[T]$ to be $\formula{Coindep}[T]$.
So $\mathcal{T}_{\Lambda}(M)$ is non-empty if and only if $M$ is a matroid and in this case it contains a set-system that is isomorphic to the dual matroid $M^{*}$.
(Or \emph{equal} to the dual matroid, as long as we blur the distinction between $E$ and $\{\{e\}\colon e\in E\}$).
\end{example}

\begin{example}
\label{examplesimple}
Next we will construct a transduction that takes each matroid to its canonical simplification.
We again set \formula{Domain} to be \formula{Matroid}.
Define the formula $\formula{Parallel}[T]$ to be
\begin{multline*}
\forall X\forall Y\ ((\formula{Sing}[X]\land \formula{Sing}[Y]\land X\ne Y\land X\subseteq T\land Y\subseteq T)\to\\
\exists C\ (\formula{Circuit}[C]\land \formula{Union}[X,Y,C])).
\end{multline*}
Thus $\formula{Parallel}[T]$ is satisfied exactly when any pairwise distinct elements in $T$ form a circuit.
Note that this means that any singleton subset will satisfy $\formula{Parallel}[T]$.
Now we set $\formula{ParallelClass}[T]$ to be
\[
\sfmax_{\formula{Parallel}}[T]\land
(\formula{Sing}[T]\to \neg\formula{Circuit}[T]).
\]
So $\formula{ParallelClass}[T]$ is satisfied by the maximal parallel sets, as long as that set does not consist of a single loop element.
Now $\formula{NewElement}[T]$ is equal to $\formula{ParallelClass}[T]$, so the ground set of the output is the set of parallel classes of $M$.
We set $\formula{NewIndep}[T]$ to be
\[
\forall C\ ((C\subseteq T\land \formula{Circuit}[C])\to \formula{Parallel}[C]).
\]
We are defining \formula{NewIndep} so that a set satisfies it exactly when it contains no circuit with more than two elements.
This is exactly the same as saying that $T$ corresponds to an independent set in the canonical simplification of the matroid $M$.
Note that this is an \mso\dash transduction and not a \cmso\dash transduction, since a member of $E_{\rho}$ could be an arbitrarily large parallel class rather than a singleton set.
\end{example}

\begin{example}
\label{examplecomponent}
In this example we construct a \cmso\dash transduction that takes each matroid to its connected components.
In this case $\mathcal{T}_{\Lambda}(M)$ may contain multiple set-systems (because $M$ may have multiple connected components), which distinguishes this transduction from those in the previous examples.

We define $\formula{Domain}[Z]$ to be $\formula{Matroid}\land \formula{Component}[Z]$.
Thus $\formula{Domain}[Z]$ is satisfied by $M=(E,\mathcal{I})$ if and only if $M$ is a matroid and the free variable $Z$ is interpreted as a connected component of $M$.
So $\mathcal{T}_{\Lambda}(M)$ will contain one output for each connected component of $M$.

Next we set $\formula{NewElement}$ to be
\[
T\subseteq Z\land\formula{Sing}[T].
\]
So the new ground set will be $\{\{e\}\colon e\in Z\}$, which we identify with the set $Z$.
Finally, \formula{NewIndep} is $T\subseteq Z\land\formula{Indep}[T]$.
This means that the independent sets of the output will be exactly the independent subsets contained in $Z$.
So $\mathcal{T}_{\Lambda}(M)$ contains one set-system for each connected component of $M$, and those set-systems are isomorphic (or equal) to the restrictions of $M$ to each connected component.
\end{example}

\begin{example}
\label{exampleminor}
Next we will demonstrate that there is a transduction taking any matroid to its minors.
Let $M=(E,\mathcal{I})$ be a matroid.
Every minor of $M$ can be expressed as $(M/I)|S$, where $I\in \mathcal{I}$ and $S$ is disjoint from $I$.
That is, every minor can be obtained by contracting an independent set and then restricting to a subset of the remaining ground set.
A subset $X\subseteq S$ is independent in $(M/I)|S$ if and only if $X\cup I$ is independent in $M$.

Let \formula{Domain} be the formula
\[
\formula{Matroid}\land \formula{Indep}[Z_{1}]\land \formula{Disjoint}[Z_{1},Z_{2}].
\]
Note that the free variables of \formula{Domain} are $Z_{1}$ and $Z_{2}$.
Set \formula{NewElement} to be $\formula{Sing}[T]\land T\subseteq Z_{2}$.
Finally set \formula{NewIndep} to be
\[
T\subseteq Z_{2}\land \forall W\ (\formula{Union}[T,Z_{1},W]\to \formula{Indep}[W]).
\]
Let $\Lambda$ be the transduction $(\formula{Domain},\formula{NewElement},\formula{NewIndep})$.
Now the explanation in the previous paragraph shows that if $\rho$ satisfies \formula{Domain}, then $M_{\rho}$ is isomorphic to $(M/\rho(Z_{1}))|\rho(Z_{2})$.
Thus $\mathcal{T}_{\Lambda}(M)$ contains an isomorphic copy of every minor of $M$, as desired.
\end{example}

\begin{example}
\label{examplerestrictions}
In this example we create a transduction that is very similar to that in Example~\ref{exampleminor}, whose outputs are the restrictions of a matroid.
We let \formula{Domain} be $\formula{Matroid}\land Z_{2}=Z_{2}$.
(We include the dummy equality because we want $Z_{2}$ to be a free variable of \formula{Domain}.)
We use exactly the same formula \formula{NewElement} as in Example~\ref{exampleminor}, and we set \formula{NewIndep} to be $T\subseteq Z_{2}\land \formula{Indep}[T]$.
Then $\mathcal{T}_\Lambda(M)$ contains one output for each restriction of $M$ to a subset of the ground set.
\end{example}

\begin{example}
\label{examplerelaxation}
Let $M$ be a matroid with the ground set $E$.
Let $\mathcal{B}$ be the collection of bases of $M$ and let $H$ be a circuit-hyperplane of $M$.
There is a matroid with ground set $E$ and $\mathcal{B}\cup \{H\}$ as its family of bases and this is known as the \emph{relaxation} of $M$ (by $H$).

Let \formula{Domain} be the formula
$\formula{Matroid}\land\formula{CircHyp}[Z_{1}]$,
where $\formula{CircHyp}[Z_{1}]$ is $\formula{Circuit}[Z]\land\formula{Hyperplane}[Z]$.
Thus 		\formula{Domain} has a single free variable, $Z_{1}$.
Set \formula{NewElement} to be $\formula{Sing}[T]$.
Set \formula{NewIndep} to be $\formula{Indep}[T]\lor T=Z_{1}$.
If $\Lambda$ is the transduction 
\[(\formula{Domain},\formula{NewElement},\formula{NewIndep})\]
then $\mathcal{T}_{\Lambda}(M)$ contains isomorphic copies of all matroids obtained from $M$ by relaxing a circuit-hyperplane.
\end{example}

\subsection{Backwards Translation Theorem}

The usefulness of transductions comes from the fact that statements about the images of a transduction can be converted to statements about the pre-images.
To illustrate this fact we consider the transduction $\Lambda$ from Example~\ref{examplecomponent}, which takes any matroid to its connected components.
Imagine we have a property of matroids that is closed under direct sums and this property is defined by the monadic sentence $\varphi$ for connected matroids.
Now the Backwards Translation Theorem allows us to lift this sentence to a formula $\varphi^{\Lambda}$.
This formula will have a single free variable and satisfying interpretations will take this variable to connected components.
By applying the universal quantifier to this free variable we arrive at a sentence which is true for a matroid exactly when all the connected components of that matroid satisfy $\varphi$.
Thus we will have arrived at a sentence which defines the property for all matroids.

We move towards a proof of the Backwards Translation Theorem, which is an analogue of~\cite{MR3776760}*{Lemma B.1}.
Both that result and Theorem~\ref{BTT} can be derived from~\cite{MR2962260}*{Theorem 7.10}
but we provide a sketch proof for the specific context of matroidal transductions.

Let \[\Lambda= (\formula{Domain}, \formula{NewElement}, \formula{NewIndep})\] be a transduction, where $\mathcal{Z}$ is the set of free variables of \formula{Domain}.
Let $\mathcal{Y}\cup\{T\}$ be the set of free variables of \formula{NewElement}, where $\mathcal{Y}\subseteq \mathcal{Z}$.
We construct the new formula $\formula{ElementUnion}[\mathcal{Z}\cup \{W\}]$ as follows:
\begin{multline*}
\formula{Domain}[\mathcal{Z}]\land
(\forall X\ ((\sing{X}\land X\subseteq W)\to\\
\exists T\ (\formula{NewElement}[\mathcal{Y}\cup\{T\}]\land X\subseteq T\land T\subseteq W)))
\end{multline*}
Let $\rho\colon \mathcal{Z}\to 2^{E}$ be a function.
If $\rho\cup\{(W,F)\}$ is a satisfying assignment of \formula{ElementUnion} it means that $\rho$ is a satisfying assignment of \formula{Domain} and $F$ is equal to a union of a sets satisfying \formula{NewElement}.
In other words, \formula{ElementUnion} is satisfied if and only if $F$ can be expressed in the form $\Lambda^{-1}_{\rho}(I)$ for some subset $I\subseteq E_{\rho}$.

\begin{theorem}[Backwards Translation Theorem]
\label{BTT}
Let
\[\Lambda=(\formula{Domain},\formula{NewElement},\formula{NewIndep})\] be an \mso\dash transduction where $\mathcal{Z}$ is the set of free variables of \formula{Domain}.
Let $\varphi$ be an \mso\dash formula with $\mathcal{X}$ as its set of free variables, where $\mathcal{Z}\cap\mathcal{X}=\emptyset$.
There exists an \mso\dash formula $\varphi^{\Lambda}$ with $\mathcal{Z}\cup\mathcal{X}$ as its set of free variables such that for every set-system $M=(E,\mathcal{I})$ and every interpretation \[\theta\colon \mathcal{Z}\cup\mathcal{X}\to 2^{E},\] $\varphi^{\Lambda}$ is satisfied by $(M,\theta)$ if and only if the following conditions hold:
\begin{enumerate}[label=$\bullet$]
\item the interpretation $\rho=\theta|_{\mathcal{Z}}$ satisfies \formula{Domain}, and
\item  for each $X\in\mathcal{X}$, there is a subset $Y_{X}\subseteq E_{\rho}$ such that $\theta(X)=\Lambda_{\rho}^{-1}(Y_{X})$, and furthermore $\varphi$ is satisfied by $M_{\rho}$ under the interpretation that takes each $X\in\mathcal{X}$ to $Y_{X}$.
\end{enumerate}
Moreover, if $\Lambda$ is a \cmso\dash transduction, then for every \cmso\dash formula $\varphi$ there exists a \cmso\dash formula $\varphi^{\Lambda}$ with the above properties.
\end{theorem}

We remark that if $\varphi$ is a sentence then $\mathcal{X}$ is empty, and $\varphi^{\Lambda}$ will be a formula with $\mathcal{Z}$ as its family of free variables.

\begin{proof}[Proof of \textup{Theorem~\ref{BTT}}.]
Let the set of free variables in \formula{NewElement} be $\mathcal{Y}_{E}\cup\{T\}$ where $\mathcal{Y}_{E}$ is a subset of $\mathcal{Z}$.
Similarly, let the set of free variables of \formula{NewIndep} be $\mathcal{Y}_{I}\cup\{T\}$ where $\mathcal{Y}_{I}\subseteq \mathcal{Z}$.
The proof is by induction on the number of steps required to build the formula $\varphi$.

Assume that $\varphi$ is atomic.
First we consider the case that $\varphi$ is \indep{X}, so that $\mathcal{X}=\{X\}$.
We set $\varphi^{\Lambda}$ to be
\[\formula{Domain}\land \formula{ElementUnion}[W\rightharpoonup X]\land \formula{NewIndep}[T\rightharpoonup X].\]
Note that the set of free variables in $\varphi^{\Lambda}$ is $\mathcal{Z}\cup \{X\}$, as desired.

Consider the function $\theta\colon \mathcal{Z}\cup \{X\}\to 2^{E}$.
We first assume that $\theta$ is a satisfying interpretation of $\varphi^{\Lambda}$.
Then $\rho=\theta|_{\mathcal{Z}}$ is a satisfying interpretation of \formula{Domain}.
Moreover, the presence of the formula $\formula{ElementUnion}[W\rightharpoonup X]$ in $\varphi^{\Lambda}$ means that $\theta(X)$ has to be a union of sets that satisfy \formula{NewElement}.
Thus there is a subset $Y_{X}\subseteq E_{\rho}$ such that $\theta(X)=\Lambda_{\rho}^{-1}(Y_{X})$.
Because $\formula{NewIndep}[T\rightharpoonup X]$ is satisfied, it follows that $Y_{X}$ is in $\mathcal{I}_{\rho}$.
So then $\varphi=\formula{Indep}[X]$ is satisfied by $M_{\rho}$ under the interpretation $X\mapsto Y_{X}$.

For the other direction, assume that $\rho=\theta|_{\mathcal{Z}}$ satisfies \formula{Domain} and that there exists $Y_{X}\subseteq E_{\rho}$ such that $\theta(X)=\Lambda_{\rho}^{-1}(Y_{X})$, where $\varphi$ is satisfied by $M_{\rho}$ under the interpretation $X\mapsto Y_{X}$. 
This means that $\formula{ElementUnion}[W\rightharpoonup X]$ is satisfied by the restriction of $\theta$ to $\{X\}$.
Moreover, because $\varphi=\indep{X}$ is satisfied by $M_{\rho}$ it follows that $Y_{X}$ belongs to $\mathcal{I}_{\rho}$, which can only be true if $\formula{NewIndep}[T\rightharpoonup X]$ is satisfied by $M$ under the interpretation
\[\rho|_{\mathcal{Y}_{I}}\cup (X,\Lambda_{\rho}^{-1}(Y_{X})) = \rho|_{\mathcal{Y}_{I}}\cup (X,\theta(X)).\]
Now the construction of $\varphi^{\Lambda}$ means that it is satisfied by $\theta$, so we have shown that the theorem holds in the case that $\varphi=\formula{Indep}(X)$.

Next we assume that $\varphi$ is $U\subseteq V$.
We set $\varphi^{\Lambda}$ to be
\[\formula{Domain}\land \formula{ElementUnion}[W\rightharpoonup U]\land\formula{ElementUnion}[W\rightharpoonup V]\land U\subseteq V.\]
Thus the set of free variables of $\varphi^{\Lambda}$ is $\mathcal{Z}\cup\{U,V\}$.
Essentially the same arguments as the previous paragraphs show that $\varphi^{\Lambda}$ has the desired properties.

Next assume that $\Lambda$ is a \cmso\dash transduction and that $\varphi$ is $|X|_{p,q}$.
 We set $\varphi^{\Lambda}$ to be
\[\formula{Domain}\land \formula{ElementUnion}[W\rightharpoonup X]\land |X|_{p,q}.\]
The fact that every member of $E_{\rho}$ is a singleton set means that $|I|=|\Lambda_{\rho}^{-1}(I)|$ for every $I\subseteq E_{\rho}$.
From this we can deduce that $\varphi^{\Lambda}$ is the desired formula using the same arguments as before.

Next we can assume that $\varphi$ is not atomic.
Let $\varphi$ be $\neg\psi$.
Then we set $\varphi^{\Lambda}$ to be
\[
\neg\psi^{\Lambda}\land\formula{Domain}\land \bigwedge_{X\in \mathcal{X}}\formula{ElementUnion}[W\rightharpoonup X].
\]
If $\varphi$ is $\varphi_{1}\land \varphi_{2}$
then we set $\varphi^{\Lambda}$ to be
\[
\varphi_{1}^{\Lambda}\land \varphi_{2}^{\Lambda}\land\formula{Domain}\land \bigwedge_{X\in \mathcal{X}}\formula{ElementUnion}[W\rightharpoonup X].
\]
Finally, if $\varphi=\exists U\psi$, then we set $\varphi^{\Lambda}$ to be
\[
\exists U \psi^{\Lambda}\land\formula{Domain}\land \bigwedge_{X\in \mathcal{X}\cup \{U\}}\formula{ElementUnion}[W\rightharpoonup X].
\]
In any of these cases it is fairly straightforward to check that $\varphi^{\Lambda}$ is the desired formula.
\end{proof}

\begin{corollary}
\label{BTTcorollary}
Let $\Lambda$ be an \mso\dash transduction.
Let $\mathcal{M}$ be an \mso\dash definable class of set-systems.
Then
\[
\{M\ \text{a set-system}\colon M_{\rho} \in \mathcal{M}\ \text{for all}\ M_{\rho}\in \mathcal{T}_{\Lambda}(M)\}
\]
and
\[
\{M\ \text{a set-system}\colon M_{\rho} \in \mathcal{M}\ \text{for at least one}\ M_{\rho}\in \mathcal{T}_{\Lambda}(M)\}
\]
are both \mso\dash definable classes.
Similarly, if $\mathcal{M}$ is \cmso\dash definable and $\Lambda$ is a \cmso\ transduction then both these classes are \cmso\dash definable.
\end{corollary}

\begin{proof}
Let $\Lambda$ be $(\formula{Domain},\formula{NewElement},\formula{NewIndep})$, where the free variables of \formula{Domain} are $Z_{1},\ldots, Z_{s}$.
Assume that $\varphi$ is an \mso\dash sentence that is satisfied precisely by the set-systems in $\mathcal{M}$.
Let $\varphi^{\Lambda}$ be the formula provided by Theorem~\ref{BTT}.
Thus the free variables of $\varphi^{\Lambda}$ are $Z_{1},\ldots, Z_{s}$, and $\rho$ is a satisfying interpretation of $\varphi^{\Lambda}$ if and only if $\rho$ satisfies \formula{Domain} and $M_{\rho}$ satisfies $\varphi$.
Now the sentence
\[
\forall Z_{1}\forall Z_{2}\cdots \forall Z_{s}\ (\formula{Domain}\to \varphi^{\Lambda})
\]
is satisfied by $M$ if and only if $M_{\rho}$ satisfies $\varphi$ for every $M_{\rho}$ in $\mathcal{T}_{\Lambda}(M)$.
Similarly,
\[
\exists Z_{1}\exists Z_{2}\cdots \exists Z_{s}\ (\formula{Domain}\land \varphi^{\Lambda})
\]
is satisfied by $M$ if and only if $M_{\rho}$ satisfies $\varphi$ for at least one $M_{\rho}$ in $\mathcal{T}_{\Lambda}(M)$.
\end{proof}

To illustrate the usefulness of Corollary~\ref{BTTcorollary}, we apply it to the transduction in Example~\ref{examplecomponent}.
The corollary now tells us that if $\mathcal{M}$ is a definable class of connected matroids, then the class of matroids, all of whose connected components belong to $\mathcal{M}$, is definable.
This is the content of~\cite{MR4395073}*{Lemma 3.9}.
We also see from Corollary~\ref{BTTcorollary} that there is a sentence defining the class of matroids at least one of whose connected components is in $\mathcal{M}$ .

In the next result we list various transductions from matroids to derived matroids.
Note that we identify each matroid with its isomorphism class, so that the set of minors of $M$ is finite since it contains one copy of each minor of $M$.

\begin{proposition}
\label{transductionlist}
In each of the following cases, there is an \mso\dash transduction $\Lambda$ so that for every set-system $M=(E,\mathcal{I})$, the image $\mathcal{T}_{\Lambda}(M)$ is empty if $M$ is not a matroid, and is otherwise as described below.
\begin{enumerate}[label = \textup{(\roman*)}]
\item The only member of $\mathcal{T}_{\Lambda}(M)$ is isomorphic to $M^{*}$, the dual of $M$.
\item The only member of $\mathcal{T}_{\Lambda}(M)$ is isomorphic to $\operatorname{si}(M)$, the canonical simplification of $M$.
\item $\mathcal{T}_{\Lambda}(M)$ is equal to the set of restrictions of $M$ to each of its connected components.
\item$\mathcal{T}_{\Lambda}(M)$ is equal to the set of $3$\dash connected components of $M$, when $M$ is a connected matroid, and is otherwise empty.
\item $\mathcal{T}_{\Lambda}(M)$ is equal to the set of minors of $M$.
\item $\mathcal{T}_{\Lambda}(M)$ is equal to the set of restrictions of $M$.
\item $\mathcal{T}_{\Lambda}(M)$ is equal to the set of  proper minors of $M$.
\item $\mathcal{T}_{\Lambda}(M)$ is equal to the set of matroids that can be obtained from $M$ by relaxing a circuit-hyperplane.
\end{enumerate}
The transductions in \textup{(i)}, \textup{(iii)}, \textup{(v)}, \textup{(vi)}, \textup{(vii)}, and \textup{(viii)} are \cmso\dash transductions.
\end{proposition}

\begin{proof}
Statements (i)--(iii) follow from Examples~\ref{exampledual},~\ref{examplesimple}, and~\ref{examplecomponent}.

The proof of statement (iv) is essentially contained in~\cite{MR4395073}*{Section~5}, although the language of transductions is not used there.
We describe the proof here without going into details.
Let $M=(E,\mathcal{I})$ be a connected matroid and let $A$ be a $2$\dash separating subset of $E$.
Let $B$ be $E-A$.
A \emph{wedge} (relative to $A$) is a set that is maximal amongst the $2$\dash separating subsets of $B$.
It is clear that there is a formula with free variables $A$ and $U$ that will be satisfied when $A$ is $2$\dash separating and $U$ is a wedge relative to $A$.
We let $\formula{Domain}$ be a formula expressing the fact that $M=(E,\mathcal{I})$ is a connected matroid and $A$ is a $2$\dash separating subset of $E$ with the property that $A$ is a flat and a coflat and any pair of distinct wedges relative to $A$ are disjoint, skew, and coskew.
We let $\formula{NewElement}$ be the formula that is satisfied by $A$ and all the wedges relative to $A$.
We let $\formula{NewIndep}$ be the formula satisfied by sets $X$ such that $X$ is a union of sets that are either $A$ or wedges relative to $A$, and any circuit contained in $X$ is contained in either $A$ or a wedge.
It now follows from Theorem~\cite{MR4395073}*{Proposition~5.7} that we have completed a description of a transduction from a connected matroid to its $3$\dash connected components.

Examples~\ref{exampleminor} and~\ref{examplerestrictions} provide proofs of (v) and (vi).
To prove statement (vii), we simply modify \formula{Domain} from Example~\ref{exampleminor} so that it is equal to
\[
\formula{Matroid}\land \formula{Indep}[Z_{1}]\land \formula{Disjoint}[Z_{1},Z_{2}]\land \exists U\ (\formula{Sing}[U]\land U\nsubseteq Z_{2}).
\]
This guarantees that there is an element that is not in $Z_{2}$, which is the ground set of the minor.
Example~\ref{examplerelaxation} provides a proof of statement (viii).
\end{proof}

\section{Operations on definable classes}
\label{sect:operations}

In this section we collect multiple tools for showing that classes of matroids are definable.

\begin{proposition}
\label{prop:singlematroid}
Let $M=(E,\mathcal{I})$ be a set-system.
Then $\mathcal{M}=\{M\}$ is \mso\dash definable.
\end{proposition}

\begin{proof}
For simplicity of notation, we assume that the ground set $E$ is equal to $\{1,2,\dots,n\}$ for some positive integer $n$. Our set of variables will be $\{X_1,\ldots,X_n\}\cup\{X,Y\}$.

We begin by defining several formulas. Let $\formula{GroundSet}_{E}$ be the formula
\[
\left(\bigwedge_{i=1}^n \sing{X_i}\right) \wedge 
       \left(\bigwedge_{i<j} X_i\neq X_j\right) \wedge
    \forall X\,\left(\sing{X}\to\bigvee_{i=1}^n (X=X_i)\right).
\]
Given a subset $U$ of $E$ let
\[
\formula{Equal}_U[X] = \left(\bigwedge_{i\in U} X_i\subseteq X\right)\wedge
     \forall Y\,\left((\sing{Y}\wedge Y\subseteq X)\to\bigvee_{i\in U}(Y=X_i)\right),
\]
and given a set $\mathcal{J}$ of subsets of $E$ let 
\[
\formula{Member}_\mathcal{J}[X] = \bigvee_{U\in \mathcal{J}} \formula{Equal}_U[X]. 
\]
Then
\[
\exists X_1\exists X_2\cdots\exists X_n\bigl(\formula{Groundset}_E\wedge
    \forall X(\indep{X}\leftrightarrow \formula{Member}_\mathcal{I}[X])\bigr)
\]
is a sentence defining $\mathcal{M}$.
\end{proof}

\begin{proposition}
\label{prop:operationsonclasses}
Let $\mathcal{M}_1,\mathcal{M}_2,\ldots,\mathcal{M}_n$ be (\mso/\cmso)\dash definable classes of set-systems.
Then the following classes are also (\mso/\cmso)\dash definable:
\begin{enumerate}[label = \textup{(\roman*)}]
\item
The complement of $\mathcal{M}_i$, for each $i$.
\item
$\displaystyle\bigcup_{i=1}^n \mathcal{M}_i$.
\item
$\displaystyle\bigcap_{i=1}^n \mathcal{M}_i$.
\end{enumerate}
\end{proposition}

\begin{proof}
For each $1\leq i\leq n$ let $\varphi_i$ be a sentence defining $\mathcal{M}_i$. Then
$\neg\varphi_i$ is a sentence defining the complement of $\mathcal{M}_i$.
The sentences
\[\ds\bigvee_{i=1}^n\varphi_i
\qquad\text{and}\qquad
\ds\bigwedge_{i=1}^n\varphi_i\]
define $\cup_{i=1}^n \mathcal{M}_i$ and $\cap_{i=1}^n\mathcal{M}_i$ respectively.
\end{proof}

The next result combines Propositions~~\ref{prop:singlematroid} and~\ref{prop:operationsonclasses}.

\begin{corollary}
\label{finite-class}
Any finite class of set-systems is \mso\dash definable.
\end{corollary}

\begin{proposition}
Let $\mathcal{M}$ be an (\mso/\cmso)\dash definable class of matroids.
The class of matroids that have at least one member of $\mathcal{M}$ as a minor is (\mso/\cmso)\dash definable.
\end{proposition}

\begin{proof}
By Proposition~\ref{transductionlist}~(v) there is a \cmso\dash transduction $\Lambda$ such that for each matroid $M$, $\mathcal{T}_\Lambda(M)$ contains a matroid isomorphic to $N$ if and only if $N$ is isomorphic to a minor of $M$. The result now follows by Corollary~\ref{BTTcorollary}.
\end{proof}

\begin{proposition}
\label{definable-excluded}
Let $\mathcal{M}$ be a minor-closed class of matroids.
Then $\mathcal{M}$ is (\mso/\cmso)\dash definable if and only if the class of excluded minors for $\mathcal{M}$ is (\mso/\cmso)\dash definable.
\end{proposition}

\begin{proof}
Assume $\mathcal{M}$ is definable via the sentence $\psi$.
Then the complement of $\mathcal{M}$ is defined by the sentence $\neg \psi$.
Now $M$ is an excluded minor if and only if it is minor-minimal in this complement; that is, $M$ does not satisfy $\psi$ but all of its proper minors do.
Let $\Lambda$ be the transduction provided by Proposition~\ref{transductionlist}~(vii), so that $\mathcal{T}_{\Lambda}(M)$ will contain an isomorphic copy of each proper minor of $M$.
Now Corollary~\ref{BTTcorollary} implies that there is a sentence $\varphi$ so that $M$ satisfies $\varphi$ if and only if every proper minor of $M$ satisfies $\psi$.
Then the sentence that defines the excluded minors of $\mathcal{M}$ is $\neg\psi\land \varphi$.

For the converse, assume that the family of excluded minors is defined by the \mso\dash sentence $\psi$.
By using Proposition~\ref{transductionlist}~(v) and Corollary~\ref{BTTcorollary} we see that there is a sentence which is satisfied by a matroid $M$ if and only if every minor of $M$ satisfies $\neg\psi$.
Then $M$ will satisfy this sentence if and only if it has no minor isomorphic to an excluded minor, which is true if and only if $M$ is in $\mathcal{M}$.
\end{proof}

Now we can deduce the known fact that any minor-closed class of matroids with only finitely many excluded minors is \mso\dash definable.

For convenience, we explicitly state the following consequence of Corollary~\ref{BTTcorollary} and Proposition~\ref{transductionlist}~(i).

\begin{corollary}
\label{dual-definable}
Let $\mathcal{M}$ be a class of matroids.
Then $\mathcal{M}$ is \mso\dash definable if and only if $\{M^{*}\colon M\in\mathcal{M}\}$ is \mso\dash definable.
The same statement applies for \cmso\dash definability.
\end{corollary}

\subsection{Sums}

Recall that if $M_{1}=(E_{1},\mathcal{I}_{1})$ and $M_{2}=(E_{2},\mathcal{I}_{2})$ are matroids with disjoint ground sets, then
\[M_{1}\oplus M_{2}=(E_{1}\cup E_{2},\{I_{1}\cup I_{2}\colon I_{1}\in\mathcal{I}_{1},\ I_{2}\in\mathcal{I}_{2}\})\]
is the \emph{direct sum} of $M_{1}$ and $M_{2}$.

\begin{proposition}
Let $\mathcal{M}_{1}$ and $\mathcal{M}_{2}$ be (\mso/\cmso)\dash definable classes.
Then \[\mathcal{M}_{1}\oplus \mathcal{M}_{2}:=\{M_{1}\oplus M_{2}\colon M_{1}\in \mathcal{M}_{1},\ M_{2}\in\mathcal{M}_{2}\}\] is (\mso/\cmso)\dash definable.
\end{proposition}

\begin{proof}
For $i=1,2$ let $\varphi_i$ be a sentence defining $\mathcal{M}_i$, and let $\Lambda$ be the \cmso\dash transduction $(\formula{Domain},\formula{NewElement},\formula{NewIndep})$ of Example~\ref{examplerestrictions} such that $\mathcal{T}_\Lambda(M)$ is the set of restrictions of $M$ to subsets of the ground set.
Note that the only free variable of \formula{Domain} is $Z_{2}$, which is interpreted as the subset we restrict to.
Thus $\varphi_{i}^{\Lambda}$ is a formula with the single free variable $Z_{2}$, for $i=1,2$.

Define the formula $\formula{DisjointSum}[X_1,X_2]$ by
\[
\formula{Bipartition}[X_1,X_2] \wedge \formula{Separator}[X_1] \wedge \formula{Separator}[X_2].
\]
Then $\formula{DisjointSum}[X_1,X_2]$ is satisfied if and only if $M$ is the sum of its restrictions to $X_1$ and $X_2$. Now define $\psi$ by
\[
\exists X_1\exists X_2\ (\formula{DisjointSum}[X_1,X_2]\land \varphi_1^\Lambda[Z_{2}\rightharpoonup X_{1}] \land\varphi_2^\Lambda[Z_{2}\rightharpoonup X_{2}]).
\]
We note that $\varphi_{i}^\Lambda[Z_{2}\rightharpoonup X_{i}]$ is satisfied when the restriction to $X_{i}$ satisfies $\varphi_{i}$.
Thus $\psi$ is a sentence defining $\mathcal{M}_1\oplus\mathcal{M}_2$.
\end{proof}

\subsection{Extensions and coextensions}

We let the formula $\formula{Coloop}[X]$ be
\[
\formula{Sing}[X]\land \forall B\ (\formula{Basis}[B]\to X\subseteq B).
\]
In the case that $M$ is a matroid, $\formula{Coloop}[X]$ is satisfied if and only if $X$ is a singleton set containing a coloop of $M$.

Let $M$ be a matroid with ground set $E$.
An element $e\in E$ is \emph{free} in $M$ if it is not a coloop and the only circuits that contain $e$ are spanning circuits.
In this case we say that $M$ is obtained from $M\backslash e$ by a \emph{free extension}.
If $e$ is free in $M^{*}$, then $M$ is obtained from $M/e$ by a \emph{cofree coextension}.
Note that in this case, $M$ is produced from $M/e$ by dualising, applying a free extension, and then dualising again.
Let $\formula{Free}[X]$ be the formula
\[
\formula{Sing}[X]\land\neg\formula{Coloop}[X]\land \forall C\ ((\formula{Circuit}[C]\land X\subseteq C)\to \formula{Spanning}[C]).
\]
If $M$ is a matroid, $\formula{Free}[X]$ will be satisfied if and only if $X$ is a singleton set containing a free element.

\begin{proposition}
\label{transduction-extensions}
In the following cases there is a \cmso\dash transduction $\Lambda$ such that for every set-system $M=(E,\mathcal{I})$ the image $\mathcal{T}_{\Lambda}(M)$ is empty if $M$ is not a matroid, and is otherwise as described below.
\begin{enumerate}[label = \textup{(\roman*)}]
\item $\mathcal{T}_{\Lambda}(M)$ is the set of matroids that can be produced from $M$ by deleting a coloop.
\item $\mathcal{T}_{\Lambda}(M)$ is the set of matroids that can be produced from $M$ by deleting a free element.
\end{enumerate}
\end{proposition}

\begin{proof}
We start with (i).
Let \formula{Domain} be
\[
\formula{Matroid}\land \exists W\ (\formula{Coloop}[W]\land \formula{Bipartition}[W,Z]).
\]
This formula will be satisfied exactly when $M$ is a matroid and $Z$ is the complement of a coloop.
Now we set $\formula{NewElement}$ to be $T\subseteq Z\land \formula{Sing}[T]$ and we set $\formula{NewIndep}$ to be $T\subseteq Z\land \formula{Indep}[T]$.
If $\rho$ is a satisfying assignment for $\formula{Domain}$, then $M_{\rho}$ is the matroid produced from $M$ by restricting to $Z$, which is to say, the matroid produced from $M$ by deleting the unique element not in $Z$.
This proves statement (i).
We prove (ii) by simply changing $\formula{Domain}$ by substituting $\formula{Free}$ for $\formula{Coloop}$.
\end{proof}

Let $C$ be a circuit of the matroid $M$.
Assume that whenever $C'$ is a circuit such that $C'\ne C$ and $C'\cap C\ne \emptyset$, then $C'$ is spanning.
In this case we say that $C$ is \emph{freely placed}.

\begin{proposition}
\label{Bn-transduction}
There is a \cmso\dash transduction $\Lambda$ so that $\mathcal{T}_{\Lambda}(M)$ is non-empty only if $M$ is a matroid with at least one freely placed and coindependent non-spanning circuit containing at least two elements and in this case, $\mathcal{T}_{\Lambda}(M)$ is exactly the set of matroids obtained from $M$ by deleting such a circuit.
\end{proposition}

\begin{proof}
We set \formula{Domain} to be
\begin{align*}
\formula{Matroid}&\land \exists C_{1}\ (\formula{Bipartition}[Z,C_{1}]\land \formula{Circuit}[C_{1}]\land \neg\formula{Spanning}[C_{1}]\land\formula{Coindep}[C_{1}]\\
&\land\forall W\ (\formula{Sing}[W]\to W\ne C_{1})\\
&\land\forall C_{2}\ ((\formula{Circuit}[C_{2}]\land\neg\formula{Disjoint}[C_{1},C_{2}]\land C_{1}\ne C_{2})\to \formula{Spanning}[C_{2}])).
\end{align*}
We then set $\formula{NewElement}$ to be $T\subseteq Z\land \formula{Sing}[T]$ and we set $\formula{NewIndep}$ to be $T\subseteq Z\land \formula{Indep}[T]$.
\end{proof}

\section{Definable classes}
\label{sect:classes}

In this section we are going to use our assembled tools to show that certain natural classes of matroids can be defined by sentences in \mso.

\subsection{Lattice-path matroids}

The class of lattice-path matroids was introduced by Bonin, de Mier, and Noy~\cite{MR2018421}.
Its importance arises from the fact that it is a subclass of transversal matroids that is closed under duality and taking minors.
(The entire class of transversal matroids is closed under neither of these operations.)
For any non-negative integer $n$ let $\{E,N\}^{n}$ be the set of words of $n$ characters chosen from the alphabet $\{E,N\}$.
Let $P=p_{1}p_{2},\ldots, p_{n}$ be any such word.
If $i\in \{1,2,\ldots, n\}$ then $n_{i}(P)$ is the number of $N$\dash characters in the first $i$ characters of $P$.
We let $N(P)$ be $\{i\in\{1,2,\ldots, n\}\colon p_{i}=N\}$.
We can identify $P$ with a walk in the integer grid from $(0,0)$ to $(n-|N(P)|,|N(P)|)$ using North and East integer steps.
Let $P$ and $Q$ be words in $\{E,N\}^{n}$ such that $|N(P)|=|N(Q)|$.
We write $P\preccurlyeq Q$ to mean that $n_{i}(P)\leq n_{i}(Q)$ for each $i=1,2,\ldots, n$.
If $L$ is another word in $\{E,N\}^{n}$ such that $|N(L)|=|N(P)|=|N(Q)|$ then we say $L$ is an \emph{intermediate path} if $P\preccurlyeq L\preccurlyeq Q$.
The family
\[
\{N(L)\colon L\ \text{is an intermediate path}\}
\]
is the family of bases of a matroid on the ground set $\{1,2,\ldots, n\}$.
We denote this matroid by $M[P,Q]$.
The matroids that arise in this way are exactly the \emph{lattice-path matroids}.

The family of lattice-path matroids is closed under minors~\cite{MR2215428}*{Theorem~3.1}, but has infinitely many excluded minors.
Nonetheless, the excluded minors have been characterised by Bonin~\cite{MR2718679} and we now describe them.

When $n\geq 2$, we define $P_{n}$ to be $T_{n}(U_{n-1,n}\oplus U_{n-1,n})$, the truncation to rank $n$ of the direct sum of two $n$\dash element circuits.
We can also define $P_{n}$ by saying that it is a rank\dash $n$ matroid with a ground set that is partitioned into two circuit-hyperplanes, where any non-spanning circuit is equal to one of these circuit-hyperplanes.
The rank-two matroid $P_{2}$ is the direct sum $U_{1,2}\oplus U_{1,2}$.
It is easy to see that $P_{n}$ is equal to its own dual.

\begin{proposition}
\label{Pn-definable}
Let $K \geq 2$ be an integer.
The class $\{P_{n}\colon n\geq K\}$ is \mso\dash definable.
\end{proposition}

\begin{proof}
The \mso\dash sentence that defines $\{P_{n}\colon n\geq K\}$ will be a conjunction with one term equal to \formula{Matroid}, so we may as well assume that the set-system $M=(E,\mathcal{I})$ is a matroid.
We can require that $M$ has rank at least $K$ by asserting that there exists a basis that contains $K$ pairwise disjoint singleton sets, so now we will assume that $r(M)\geq K$.
We complete the sentence by taking the conjunction with the following:
\begin{multline*}
\exists X_{1}\exists X_{2}\ (\formula{Bipartition}[X_{1},X_{2}]\land \formula{CircHyp}[X_{1}]\land \formula{CircHyp}[X_{2}]\land\\
\forall C\ ((\formula{Circuit}[C]\land \neg\formula{Spanning}[C])\to C=X_{1}\lor C=X_{2})).\qedhere
\end{multline*}
\end{proof}

Let $M$ be a rank\dash $n$ matroid where $n\geq 3$ and let $e$ and $f$ be distinct elements of $E(M)$ such that $e$ is free in $M$, $f$ is free in $(M\backslash e)^{*}$, and $(M\backslash e)/f$ is isomorphic to $P_{n-1}$.
In this case we denote $M$ by $A_{n}$.
The matroid $A_{3}$ is also known as $Q_{6}$.
It is easy to see that $A_{n}$ is isomorphic to its own dual.

\begin{proposition}
\label{An-definable}
The class $\{A_{n}\colon n\geq 3\}$ is \mso\dash definable.
\end{proposition}

\begin{proof}
We rely on Proposition~\ref{transductionlist}~(i) and Proposition~\ref{transduction-extensions}~(ii).
Let $\Lambda_{1}$ be the transduction which takes each matroid to its dual, and let $\Lambda_{2}$ be the transduction which takes each matroid to the set of minors produced by deleting a single free element.
Let $\varphi$ be the sentence provided by Proposition~\ref{Pn-definable}, where a set-system will satisfy $\varphi$ if and only if it is isomorphic to $P_{n}$ for some $n\geq 2$.
Now $M=(E,\mathcal{I})$ satisfies the sentence $((\varphi^{\Lambda_{2}})^{\Lambda_{1}})^{\Lambda_{2}}$ if and only if $M$ is a  matroid and we can produce a matroid in $\{P_{n}\colon n\geq 2\}$ by first deleting a free element, then taking the dual, and then deleting another free element.
Because $P_{n}=P_{n}^{*}$ this is equivalent to saying that $M$ can be produced from some $P_{n}$ for $n\geq 2$ by a cofree coextension and then a free extension.
This is exactly what it means for $M$ to be isomorphic to $A_{n}$ for some $n\geq 3$.
\end{proof}

Let $n$ and $k$ be integers satisfying $n\geq k\geq 2$.
Then $B_{n,k}$ is the matroid $T_{n}(U_{n-1,n}\oplus U_{n-1,n}\oplus U_{k-1,k})$, where $T_{n}(M)$ denotes the truncation to rank $n$ of the matroid $M$.

\begin{proposition}
\label{Bn-prop}
A matroid $M$ is isomorphic to $B_{n,k}$ for some choice of $n$ and $k$ if and only if it has a non-spanning circuit $C$ such that $C$ is freely placed and coindependent, $|C|\geq 2$, and $M\backslash C$ is isomorphic to $P_{n}$ for some value $n\geq 2$.
\end{proposition}

\begin{proof}
Let $N_{1}$ and $N_{2}$ be isomorphic copies of $U_{n-1,n} $ on the ground sets $E_{1}$ and $E_{2}$, respectively.
Let $N_{3}$ be a copy of $U_{k-1,k}$ on $E_{3}$, where $E_{1}$, $E_{2}$, and $E_{3}$ are pairwise disjoint.
Note that $|E_{3}|=k\geq 2$.
Let $M$ be the truncation to rank $n$ of the direct sum $N_{1}\oplus N_{2}\oplus N_{3}$, so that $M$ is isomorphic to $B_{n,k}$.
Now $E_{3}$ is a circuit of $N_{1}\oplus N_{2}\oplus N_{3}$, so we can easily show that it is a circuit of $M$.
The non-spanning circuits of the rank\dash $n$ truncation are exactly the circuits with size at most $n$ in $N_{1}\oplus N_{2}\oplus N_{3}$, and these are exactly $E_{1}$, $E_{2}$, and $E_{3}$.
It follows that $E_{3}$ is freely placed.
Because $E_{1}\cup E_{2}$ contains a basis of $M$, it follows that $E_{3}$ is coindependent.
Moreover, $M\backslash E_{3}$ has exactly two non-spanning circuits, each of size $n$, and these circuits partition the ground set.
It follows easily that these circuits are also hyperplanes, and therefore $M\backslash E_{3}$ is isomorphic to $P_{n}$.

For the converse, assume that $C$ is a circuit of $M$ such that $C$ is freely placed and coindependent, $|C|\geq 2$, and $M\backslash C$ is isomorphic to $P_{n}$.
Let $k$ be the size of $C$.
Because $C$ is coindependent we see that $M\backslash C$ and $M$ both have rank $n$.
Let $(C_{1},C_{2})$ be the unique partition of $E(M)-C$ into two non-spanning circuits of $M\backslash C$.
The only non-spanning circuits of $M$ that do not intersect $C$ are $C_{1}$ and $C_{2}$.
The only non-spanning circuit of $M$ that does intersect $C$ is $C$ itself, because $C$ is freely placed.
Thus $C_{1}$, $C_{2}$, and $C$ are the only non-spanning circuits of $M$.
We can easily see that $T_{n}(U_{n-1,n}\oplus U_{n-1,n}\oplus U_{k-1,k})$ has the same rank as $M$ and the same set of non-spanning circuits (up to isomorphism).
Because a matroid is determined by its rank and its collection of non-spanning circuits, it now follows that $M$ is isomorphic to $B_{n,k}$.
\end{proof}

The next result follows immediately from Propositions~\ref{Bn-transduction}, \ref{Pn-definable}, and~\ref{Bn-prop}.

\begin{corollary}
\label{Bn-definable}
The class $\{B_{n,k}\colon n\geq k\geq 2\}$ is \mso\dash definable.
\end{corollary}

Let $n\geq 4$ be an integer.
Let $M$ be a matroid with distinct elements $e$ and $f$ such that $e$ is free in $M$, $f$ is a coloop in $M\backslash e$, and $M\backslash \{e,f\}$ is isomorphic to $P_{n-1}$.
In this case $M$ is the matroid $D_{n}$.
The next result follows from Propositions~\ref{transduction-extensions} and~\ref{Pn-definable}.

\begin{proposition}
\label{Dn-definable}
The class $\{D_{n}\colon n\geq 4\}$ is \mso\dash definable.
\end{proposition}

The rank-three matroid $R_{3}$ has ground set $\{1,2,\ldots, 7\}$.
Its non-spanning circuits are $\{1,2\}$, $\{3,4\}$, and any set in the family
\[
\{\{5,x,y\}\colon x\in\{1,2\}, y\in \{3,4\}\}.
\]
Recall that the rank-three wheel $\mathcal{W}_{3}$ is isomorphic to $K_{4}$, and the rank-three whirl $\mathcal{W}^{3}$ is obtained from $M(\mathcal{W}_{3})$ by declaring a circuit-hyperplane to be a basis.
Both $M(\mathcal{W}_{3})$ and $\mathcal{W}^{3}$ are self-dual.
Now we can state Bonin's characterisation of lattice-path matroids.

\begin{theorem}[\cite{MR2718679}*{Theorem~3.1}]
\label{LP-excluded-minors}
The excluded minors for the class of lattice-path matroids are:
\begin{enumerate}[label = \textup{(\roman*)}]
\item $A_{n}$ for $n\geq 3$,
\item $B_{n,k}$ and $B_{n,k}^{*}$ for $n\geq k\geq 2$,
\item $D_{n}$ and $D_{n}^{*}$ for $n\geq 4$,
\item $M(\mathcal{W}_{3})$, $\mathcal{W}^{3}$, $R_{3}$, and $R_{3}^{*}$.
\end{enumerate}
\end{theorem}

\begin{theorem}
\label{LP-definable}
The class of lattice path matroids is \mso\dash definable.      
\end{theorem}

\begin{proof}
By Proposition~\ref{definable-excluded} it suffices to show that the set of excluded minors for the class of lattice-path matroids is \mso\dash definable.
Propositions~\ref{An-definable}, \ref{Bn-definable}, and~\ref{Dn-definable} along with Corollary~\ref{dual-definable} imply that the classes in (i), (ii), and (iii) of Theorem~\ref{LP-excluded-minors} are \mso\dash definable.
The class of four matroids in (iv) is \mso\dash definable by Corollary~\ref{finite-class}.
Now the union of these four classes is \mso\dash definable by Proposition~\ref{prop:operationsonclasses}~(ii).
\end{proof}

\subsection{Spikes}

The class of spikes has played a central role in matroid theory since their introduction in~\cite{MR1399683}.
We define them in the following way.
Let $r\geq 3$ be an integer.
Let $L_{1},\ldots, L_{r}$ be a collection of pairwise disjoint subsets,
each of cardinality two.
We consider a rank\dash $r$ matroid on the ground set $\cup_{i=1}^{r}L_{i}$.
We call the sets $L_{1},\ldots, L_{r}$ the \emph{legs} of the matroid.
Every set of the form $L_{i}\cup L_{j}$ (where $i\ne j$) is a circuit.
Any non-spanning circuit that is not of this form must be a circuit-hyperplane that intersects each leg in exactly one element
(but there may not be any such circuits).
If these conditions are satisfied then the matroid is a \emph{(tipless) spike}.

\begin{proposition}
The class of spikes is \mso\dash definable.
\end{proposition}

\begin{proof}
A spike with rank three or four has at most eight elements, so there are only finitely many such spikes.
Corollary~\ref{finite-class} implies that there is an \mso\dash sentence
$\varphi_{\leq 4}$ that defines this subclass of spikes.
Hence we need only consider spikes with rank at least five.
This constraint means that the only circuits of size four are the unions of two distinct legs.
We define the \mso\dash formula \formula{Leg}[X] so that it is satisfied when $X$ has cardinality two and there exist two circuits $C_{1}$ and $C_{2}$ such that $|C_{1}|=|C_{2}|=4$ and $X=C_{1}\cap C_{2}$.
We refer to the sets that satisfy \formula{Leg} as \emph{legs}.
Now we construct the sentence \formula{Spike} to be the conjunction of $\varphi_{\leq 4}$ with a sentence saying that every singleton set is contained in exactly one leg, that every union of two distinct legs is a circuit, and that every non-spanning circuit not of this form is a circuit-hyperplane that intersects each leg in exactly one element.
Thus \formula{Spike} is satisfied exactly by the tipless spikes.
\end{proof}

We will now consider the smallest minor-closed class of matroids that contains all spikes.
Let this class be denoted by $\mathcal{S}$.
Problem 3.11 in~\cite{MR4224059} conjectures that $\mathcal{S}$ has only finitely many excluded minors.
In this section we prove the weaker result that $\mathcal{S}$ is \mso\dash definable.

The class $\mathcal{S}$ is fairly simple to describe explicitly.
The easiest method for doing so involves \emph{lift matroids}, as introduced by Zaslavsky~\cite{MR1088626}.
Let $G$ be a graph (which may contain loops and parallel edges).
A \emph{theta subgraph} consists of two distinct vertices joined by three paths with no internal vertices in common.
A \emph{linear class} is a collection $\mathcal{B}$ of cycles with the property that no theta subgraph contains exactly two cycles in $\mathcal{B}$.
Any cycle contained in $\mathcal{B}$ is \emph{balanced} and any other cycle is \emph{unbalanced}.
If a subgraph contains only unbalanced cycles then it is \emph{contrabalanced}.

Consider a matroid with the set of edges of $G$ as its ground set.
The circuits of this matroid are the subsets of edges that correspond to:
\begin{enumerate}[label = \textup{(\roman*)}]
\item balanced cycles,
\item contrabalanced theta subgraphs, or
\item a pair of edge-disjoint unbalanced cycles that have at most one vertex in common.
\end{enumerate}
We write $L(G,\mathcal{B})$ to denote this matroid and we say that $L(G,\mathcal{B})$ is a \emph{lift matroid}.
If $k_{G}$ is the number of connected components of $G$ that have no unbalanced cycles then the rank of $L(G,\mathcal{B})$ is $|V(G)|-k_{G}$.

Let $\mathcal{G}$ be the class of graphs containing:
\begin{enumerate}[label=\textup{(\roman*)}]
\item any connected graph with exactly two vertices and at most four edges joining them,
\item any graph whose underlying simple graph is a cycle of at least three
vertices, where each parallel class contains at most two edges.
\end{enumerate}
Every spike is a lift matroid of the form $L(G,\mathcal{B})$, where $G\in \mathcal{G}$ is a loopless graph with at least three vertices, every edge is in a parallel class of size two, and $\mathcal{B}$ is a linear class of Hamiltonian cycles.

Our next result characterises the matroids in $\mathcal{S}$.
It is a consequence of~\cite{MR4224059}*{Proposition~3.5}.

\begin{lemma}
\label{spike-minors}
Let $M$ be a matroid.
Then $M$ belongs to $\mathcal{S}$ if and only if one of the following statements holds.
\begin{enumerate}[label = \textup{(\roman*)}]
\item $M = L(G,\mathcal{B})$, where $G\in \mathcal{G}$ has at least
three vertices and $\mathcal{B}$ is a linear class of Hamiltonian cycles,
\item $M = L(G,\mathcal{B})$, where $G\in \mathcal{G}$ has exactly two vertices and $\mathcal{B}$ is a collection of pairwise edge-disjoint Hamiltonian cycles, 
\item $M=M(G)$ for some graph $G\in \mathcal{G}$,
\item $M=M^{*}(G)$ for some graph $G\in \mathcal{G}$,
\item every connected component of $M$ has size at most two, or
\item $M$ has rank at most one.
\end{enumerate}
\end{lemma}

\begin{lemma}
\label{lem:spikeclass}
Let $M$ be a matroid.
Conditions \textup{(A)} and \textup{(B)} are equivalent.
\begin{enumerate}[label=\textup{(\Alph*)}]
\item There exists a graph $G\in\mathcal{G}$ such that $G$ has at least five vertices and at least three parallel pairs and $M=L(G,\mathcal{B})$ (where $\mathcal{B}$ is a linear class of Hamiltonian cycles).
\item There exist disjoint subsets $P$ and $S$ such that if we define $Z$ to be $E(M)-(P\cup S)$ then the following conditions hold:
\begin{enumerate}[label = \textup{(\roman*)}]
\item $r(M)\geq 5$ and $|Z|\geq 6$,
\item if $|C|=2$ then $C$ is a circuit if and only if $C\subseteq P$,
\item for every $z\in Z$ there exists a unique element $z^{*}\in Z-z$ such that there are at least two $4$\dash element circuits of $M|Z$ that contain $\{z,z^{*}\}$, and furthermore
\begin{enumerate}[label = \textup{(\alph*)}]
\item there exists a basis $B$ of $M$ such that $S\subseteq B\subseteq S\cup Z$ and $B\cap \{z,z^{*}\}=1$ for every $z\in Z$,
\item $\{z,z^{*},p\}$ is a circuit for any $z\in Z$ and any $p\in P$,
\item if $z_{1}$ is in $Z$ and $z_{2}$ is in $Z-\{z_{1},z_{1}^{*}\}$, then $\{z_{1},z_{1}^{*},z_{2},z_{2}^{*}\}$ is a circuit, and
\end{enumerate}
\item if $C$ is a non-spanning circuit then either:
\begin{enumerate}[label = \textup{(\alph*)}]
\item $|C|=2$ and $C\subseteq P$,
\item $C=\{z,z^{*},p\}$ for some $z\in Z$ and some $p\in P$,
\item $C=\{z_{1},z_{1}^{*},z_{2},z_{2}^{*}\}$ for some $z_{1}\in Z$ and $z_{2}\in Z-\{z_{1},z_{1}^{*}\}$, or
\item $C$ is a circuit-hyperplane satisfying $S\subseteq C\subseteq S\cup Z$ and $|C\cap\{z,z^{*}\}|=1$ for every $z\in Z$.
\end{enumerate}
\end{enumerate}
\end{enumerate}
\end{lemma}

\begin{proof}
Assume that (A) holds, so that $M=L(G,\mathcal{B})$, where $G\in \mathcal{G}$ has at least five vertices and at least three parallel pairs and where $\mathcal{B}$ is a linear class of Hamiltonian cycles.
We set $P$ to be the set of loops of $G$ and we set $S$ to be the set of non-loop edges that are not in parallel pairs.
Next we set $Z$ to be $E(M)-(P\cup S)$.
It is easy to check that conditions (i) and (ii) in (B) hold.
Note that if $z$ is in $Z$ then $z$ is a non-loop edge that is in a parallel pair.
We set $z^{*}$ to be the edge parallel to $z$.
There are at least two $4$\dash element circuits of $M|Z$ containing $\{z,z^{*}\}$ because there are at least three parallel pairs in $Z$.
Because $\mathcal{B}$ is a linear class and there are at least three parallel pairs, we can certainly find a Hamiltonian cycle that is not in $\mathcal{B}$.
Such a Hamiltonian cycle will satisfy (iii)(a).
The remaining conditions are easy to verify, using the definition of circuits of lift matroids.

For the converse we assume that $M$ satisfies the conditions in (B).
Assume that $z$ is an element of $Z$.
Then there is a unique element $z^{*}\in Z-z$ such that there are at least two $4$\dash element circuits contained in $Z$ that contain $\{z,z^{*}\}$.
There is also a unique element $z^{**}$ such that there are at least two $4$\dash element circuits contained in $Z$ that contain $\{z^{*},z^{**}\}$.
But $z$ is such an element, so it follows that $z=z^{**}$.
Now it follows that condition (iii) partitions $Z$ into $2$\dash element subsets, $\{a_{i},b_{i}\}_{i=1}^{k}$, where $b_{i}=a_{i}^{*}$ for each $i$.
Note that $|Z|=2k$ and $k\geq 3$ by condition (i).
	
Next we construct the graph $G$.
We start with a cycle with edge-set equal to $S\cup\{a_{1},\ldots, a_{k}\}$.
We add each edge $b_{i}$ so that it is parallel to $a_{i}$.
Finally we add $P$ as a set of loops incident with an arbitrary vertex.
Assume that $C$ is a non-spanning circuit of $M$ that satisfies (iv)(d).
Then $C$ is a Hamiltonian cycle of $G$.
Let $\mathcal{B}$ be the class of Hamiltonian cycles arising in this way.
If $\mathcal{B}$ is not a linear class then there exists a theta subgraph that contains exactly two cycles in $\mathcal{B}$.
This theta subgraph necessarily comprises a Hamiltonian cycle $C$ that contains $a_{i}$ for some $i$ and another Hamiltonian cycle of the form $(C-a_{i})\cup b_{i}$.
Thus $C$ and $(C-a_{i})\cup b_{i}$ are two circuit-hyperplanes of $M$.
This implies that $C$ is not a flat of $M$ so we have a contradiction.
Therefore $\mathcal{B}$ is a linear class of $M$.

Condition (iii)(a) implies that there is a basis of $M$ corresponding to a Hamiltonian cycle of $G$.
This basis contains at least five edges by (i), so $G$ has at least five vertices and at least three parallel pairs.
Note that any cycle $\{a_{i},b_{i}\}$ of $G$ is unbalanced.
Because there is at least one unbalanced cycle it follows that the rank of $L(G,\mathcal{B})$ is equal to the number of vertices in $G$, which is equal to the number of edges in a Hamiltonian cycle.
Thus $M$ and $L(G,\mathcal{B})$ have the same rank.
By applying (iv) and the definition of circuits in a lift matroid, we can
check that $M$ and $L(G,\mathcal{B})$ have exactly the same non-spanning circuits, so they are identical matroids.
\end{proof}

\begin{theorem}
\label{thm:spike-definable}
The class $\mathcal{S}$ is \mso\dash definable.
\end{theorem}

\begin{proof}
We will, in turn, consider the classes of matroids from Lemma~\ref{spike-minors}.
Start by considering matroids of the form $L(G,\mathcal{B})$, where $G\in \mathcal{G}$ has at least three vertices.
Consider the case that $G$ has four vertices.
We can assert that the ground set is partitioned into the sets $X_{1}$, $X_{2}$, $X_{3}$, and $X_{4}$ (corresponding to the parallel classes of $G$), and $P$ (corresponding to the set of loops) so that there exists a basis consisting of a single element from each of $X_{1}$, $X_{2}$, $X_{3}$, and $X_{4}$.
We insist that $1\leq |X_{i}|\leq 2$ for each $i$ but we allow $P$ to be empty.
We furthermore assert that the following sets are circuits.
\begin{enumerate}[label = \textup{(\roman*)}]
\item $2$\dash element subsets of $P$,
\item $X_{i}\cup Y$ where $|X_{i}|=2$ and $Y$ is a singleton subset of $P$, and
\item $X_{i}\cup X_{j}$, where $i\ne j$ and $|X_{i}|=|X_{j}|=2$.
\end{enumerate}
We also require that any non-spanning circuit is either one of these sets, or is a circuit-hyperplane that contains a single element from each of $X_{1}$, $X_{2}$, $X_{3}$, and $X_{4}$.
This sentence characterises matroids of the form $L(G,\mathcal{B})$ where $G$ has four vertices.
It is clear that we can repeat this exercise when $G$ has three vertices, so now we will examine the case that $G$ has at least five vertices.

Consider the case that $G$ has at least five vertices and exactly two parallel pairs.
We assert that the ground set is partitioned into sets $S$, $X_{1}$, $X_{2}$, and $P$, where $|X_{1}|=|X_{2}|=2$ (and $P$ may be empty but $S$ must contain at least three elements).
We require that there is a basis containing $S$ and a single element from each of $X_{1}$ and $X_{2}$, and that the following sets are circuits:
\begin{enumerate}[label = \textup{(\roman*)}]
\item $2$\dash element subsets of $P$,
\item $X_{i}\cup Y$ where $|X_{i}|=2$ and $Y$ is a singleton subset of $P$, and
\item $X_{1}\cup X_{2}$.
\end{enumerate}
We impose the condition that any non-spanning circuit is one of these sets or is a circuit-hyperplane consisting of $S$ along with a single element from each of $X_{1}$ and $X_{2}$.
It is clear that we can make alterations that deal with the case that $G$ has fewer than two parallel pairs.

Now we will consider the matroids that satisfy condition (A) of Lemma~\ref{lem:spikeclass}.
We will construct a sentence that asserts that there exist disjoint subsets $P$ and $S$ so that when $Z$ is the complement of $P\cup S$ the conditions in (B) hold.
We can assert that there exists an independent union of five distinct singleton sets and that there exists a union of six distinct singleton sets in $Z$.
Therefore we can impose condition (B)(i).
It is obvious that we can construct a formula that imposes condition (B)(ii).
Let $\formula{Pair}[X,Y]$ be constructed so that it is satisfied when $X$ and $Y$ are distinct singleton subsets of $Z$ and there exist at least two $4$\dash element circuits contained in $Z$ that contain both $X$ and $Y$.
Now we will assert that for every singleton subset $X$ of $Z$ there is a unique subset $Y$ such that $\formula{Pair}[X,Y]$ is satisfied.
Certainly we can construct formulas corresponding to all the conditions in (B)(iii) and (B)(iv).
This argument shows that there is an \mso\dash sentence that defines the class of matroids satisfying (A) from Lemma~\ref{lem:spikeclass}.
Now we are done with the class of matroids satisfying (i) from Lemma~\ref{spike-minors}.

Consider an arbitrary matroid $M$ in class (ii) from Lemma~\ref{spike-minors}.
We see that $r(M)=2$ and that $M$ may have an arbitrary number of loops, but at most six parallel classes.
There are at most two parallel classes corresponding to unbalanced loops incident with the two vertices, so two of these parallel classes may have arbitrary size.
If $M$ has exactly three non-trivial parallel classes then it has no more than five parallel classes, and at least one of them has cardinality two.
(This corresponds to $\mathcal{B}$ containing exactly one Hamiltonian cycle.)
If $M$ has four non-trivial parallel classes, then it has exactly four parallel classes, and at least two of them have cardinality two.
These conditions characterise the matroids in class (ii), and it is clear that they can all be stated in \mso.

We skip ahead and show that we can define the classes in (v) and (vi).
For (v) we can state that whenever $X$ is the union of three pairwise distinct singleton sets, there is no set that contains $X$ and satisfies the formula \formula{Component}.
For (vi) all we need do is require that any union of two distinct singleton sets satisfies \formula{Dep}.

Now we are left only with the classes (iii) and (iv).
If we can show that class (iii) is \mso\dash definable, then (iv) will follow by Corollary~\ref{dual-definable}.
By the previous paragraph, we need only consider matroids in (iii) with rank at least two.
These are the graphic matroids of the form $M(G)$ where $G\in \mathcal{G}$ has at least three vertices, which is to say the matroids with no parallel class of size three or more whose canonical simplification is isomorphic to $U_{n-1,n}$ for some $n\geq 3$.
First we note that there is a sentence $\varphi$ that defines the class $\{U_{n-1,n}\colon n\geq 3\}$.
(We could, for example, construct $\varphi$ so that it asserts the ground set has size at least three and whenever $C$ satisfies \formula{Circuit}, every singleton set is a subset of $C$.)
Next we will tweak the transduction in Example~\ref{examplesimple}.
We alter \formula{Domain} so that it is the conjunction of \formula{Matroid} with a sentence saying that whenever $X$ is the union of three distinct singleton sets then $X$ does not satisfy \formula{Parallel}.
Thus \formula{Domain} is satisfied by exactly the matroids that have no parallel classes with three or more elements.
The other formulas in the transduction are exactly as in Example~\ref{examplesimple}.
Let $\Lambda$ be the transduction that we obtain in this way.
Neither $\varphi$ nor \formula{Domain} has any free variables, so when we apply Theorem~\ref{BTT} to obtain $\varphi^{\Lambda}$ we produce a formula with no free variables: that is, a sentence.
This sentence will be satisfied exactly by the matroids with no parallel class of size three or more whose canonical simplification is isomorphic to $U_{n-1,n}$ for some $n\geq 3$, which is precisely what we desired.
This disposes of cases (iii) and (iv) and shows that Theorem~\ref{thm:spike-definable} holds.
\end{proof}

\end{document}